\documentclass[11pt, reqno]{amsart}

\usepackage{amsmath,amssymb,amsthm, epsfig}
\usepackage{hyperref}
\usepackage{amsmath}
\usepackage{amssymb}
\usepackage{color}
\usepackage{ulem}
\usepackage{epsfig}
\usepackage[mathscr]{eucal}
\usepackage[latin1]{inputenc}

\usepackage[usenames,dvipsnames]{pstricks}
\usepackage{epsfig}
\usepackage{pst-grad} % For gradients
\usepackage{pst-plot} % For axes
\usepackage[space]{grffile} % For spaces in paths
\usepackage{etoolbox} % For spaces in paths
\makeatletter % For spaces in paths
\patchcmd\Gread@eps{\@inputcheck#1 }{\@inputcheck"#1"\relax}{}{}
\makeatother

\newtheorem{theorem}{Theorem}
\newtheorem{definition}{Definition}

\newtheorem{lemma}{Lemma}
\newtheorem{proposition}{Proposition}
\newtheorem{corollary}{Corollary}
\newtheorem{remark}{Remark}
\newtheorem*{example*}{Example} 
\date{}
\numberwithin{equation}{section}
\numberwithin{theorem}{section}
\numberwithin{lemma}{section}
\numberwithin{corollary}{section}
\numberwithin{remark}{section} 
\numberwithin{proposition}{section}
\numberwithin{definition}{section}

\def \R {\mathbb{R}}

\def \loc {\mathrm{loc}}

\begin{document}

\title[Fully nonlinear dead-core systems]{Fully nonlinear dead-core systems}

\author[D.J. Ara\'ujo]{Dami\~ao J. Ara\'ujo}
\address{Department of Mathematics, Federal University of Para\'iba, 58059-900, Jo\~ao Pessoa, PB, Brazil}
\email{araujo@mat.ufpb.br}

\author[R. Teymurazyan]{Rafayel Teymurazyan}
\address{King Abdullah University of Science and Technology (KAUST), Computer, Electrical and Mathematical Sciences and Engineering Division (CEMSE), Thuwal 23955-6900, Saudi Arabia and University of Coimbra, CMUC, Department of Mathematics, Largo D. Dinis, 3000-143 Coimbra, Portugal}{} 
\email{rafayel.teymurazyan@kaust.edu.sa} 
 
\begin{abstract}
	We study fully nonlinear dead-core systems coupled with strong absorption terms. We discover a chain reaction, exploiting properties of an equation along the system and obtain higher sharp regularity across the free boundary. Additionally, we prove geometric measure estimates and obtain coincidence of the free boundaries. We also derive Liouville type theorems for entire solutions. These results are new even for linear systems.

\bigskip


\bigskip

\end{abstract}

\maketitle

\tableofcontents

\section{Introduction}\label{s1}
Due to applications in population dynamics, combustion processes, catalysis processes and in industry, in particular, in bio-technologies and chemical engineering, reaction-diffusion equations with strong absorption terms have been studied intensively in recent years. In literature, the regions where the density of a certain substance (liquid, gas) vanishes, are referred to as dead-cores. When the density of one substance is influenced by that of another one, we then deal with dead-core systems. The system modeling these densities of the reactants $u$, $v:\Omega\rightarrow\R$, where $\Omega\subset\R^n$, is given by
\begin{equation}\label{fgsystem}
\left\{
\begin{array}{ccc}
F(D^2 u,x)&=&f(u,v,x),\,\,\,\text{ in }\,\,\,x\in\Omega,\\[0.2cm]
G(D^2 v,x)&=&g(u,v,x),\,\,\,\text{ in }\,\,\,x\in\Omega.
\end{array}
\right.
\end{equation}
The operators $F$, $G:\mathbb{S}^n\times\Omega\rightarrow\R$, where $\mathbb{S}^n$ is the set of symmetric matrices of order $n$, model the diffusion processes, while the functions $f$ and $g$ are the convection terms of the model. Additionally, systems like \eqref{fgsystem} are related to Lane-Emden systems (see \cite{MNS19}), two membranes problem (see \cite{CDV18,CDS17,S05}), and also model games that combine the tug-of-war with random walks in two different boards (see \cite{MR20}).

The main goal of this paper is to study the regularity of solutions for fully non-linear dead-core systems
\begin{equation}\label{1.1}
\left\{
\begin{array}{ccc}
F(D^2u,x)=v_+^{\,p}, & \text{in} & \Omega,\\[0.2cm]
G(D^2v,x)=u_+^q & \text{in} & \Omega,
\end{array} 
\right. 
\end{equation}
where $\Omega\subset\R^n$ is a bounded domain, $p$ and $q$ are non-negative constants such that $pq<1$, and $F$ and $G$ are fully nonlinear uniformly elliptic operators with uniformly continuous coefficients. The condition $pq<1$ plays the role of the sublinearity assumption (see, for example, \cite{MNS19, ST18} and references therein). In case $pq>1$, dead-core solutions cannot exist, \cite{MNS19}. Solutions of \eqref{1.1} are understood in the viscosity sense (see Definition \ref{d1.1}).

In literature (see, for example, \cite{C89,CIL92, I92,IK91, IK912}) viscosity solutions for systems are defined as locally bounded functions. Since our operators are uniformly elliptic (with continuous coefficients), using a fixed point argument, we are able to prove existence of continuous solutions. Based on this, we define viscosity solutions as continuous functions, which is in accordance with the classical definition for equations. Although existence of solutions, Perron's method, as well as several geometric properties for systems are well studied in the literature, they are valid only under certain superlinearity, monotonicity, or growth conditions, which do not apply in our framework. 

From Krylov-Safonov regularity theory (see \cite{CC95}), one may deduce that viscosity solutions of the system \eqref{1.1} are locally of the class $C^{1,\alpha}$, for $\alpha\in(0,1)$. The ideas introduced by Teixeira in \cite{T16}, where dead-core equations were studied, inspire hope for more regularity near the free boundary. They were implemented to derive higher regularity for the dead-core problems ruled by the infinity Laplacian, \cite{ALT16} (see also \cite{DT19}). The study of those equations, particularly models of the form $\Delta u \sim u^{p}$ with $p\in(0,1)$, has a deep historical background and dates back to the classical Alt-Phillips problem, \cite{AP86, P1, P2}. Unlike dead-core problems for equations, when dealing with systems, one has two main challenges: the lack of a comparison principle and the lack of the classical Perron's method in the framework. The latter for systems is valid only under certain structural assumptions on the operators (see \cite{C89,CIL92,I92,IK91,IK912}). Without these tools, even proving existence of solutions is a challenging task since, for systems, in general, the comparison of a sub- and a super-solution does not hold, \cite[Remark 3.4]{IK91}. As for the comparison principle, it is not valid even for linear systems. For example, when $F=G=\Delta$ in \eqref{fgsystem}, i.e., considering the system for the Laplace operator, the solution of which can be interpreted as the Euler-Lagrange equation of the energy functional
$$
I(u,v):=\int_\Omega\nabla u\cdot\nabla v+\mathcal{F}(u,v)\,dx \longrightarrow \mbox{min},
$$
where $f(u,v,x)=\partial_u \mathcal{F}(u,v,x)$ and $g(u,v,x)=\partial_v \mathcal{F}(u,v)$, the comparison of solutions on the boundary does not imply comparison inside the domain. 

The lack of a comparison principle does not allow the construction of suitable barriers to study the geometric properties of viscosity solutions. Nevertheless, we obtain additional information for solutions $v$ of the second equation of the system by plugging in the information about $u$ derived from the first equation (we refer to it as a chain reaction). In particular, we show that if a pair $(u,v)$ of non-negative functions is a viscosity solution of the system \eqref{1.1}, and $x_0\in\partial\{u+v>0\}$, then
$$
c\,r^{\frac{2}{1-pq}}\le\sup_{B_r(x_0)}\left(u^{\frac{1}{1+p}}+v^{\frac{1}{1+q}}\right)\le C\,r^{\frac{2}{1-pq}} 
$$
where $r>0$ is small enough and $c,\,C>0$ are constants depending on the dimension. In particular, this implies 
$$
u\in C^{\frac{2(1+p)}{1-pq}}\,\,\,\textrm{ and }\,\,\,v\in C^{\frac{2(1+q)}{1-pq}}
$$
along the set $\partial\{u+v>0\}$. Since $(1-pq)^{-1}\to\infty$, as $pq\to1$, then this regularity is higher than that coming from the Krylov-Safonov and Schauder regularity theory. Indeed, for example, for the Laplace operator, as $\Delta u\in C^{0,p}_\loc$ and $\Delta v\in C^{0,q}_\loc$, one has $u\in C^{2,p}_\loc$ and  $v\in C^{2,q}_\loc$, and coupling this into the system, we conclude that $u$, $v\in C_\loc^{2,1^-}$.

Observe that if $F=G$, $p=q$ and $u=v$ in \eqref{1.1}, we recover the regularity result obtained in \cite{T16}. It is noteworthy that the regularity we derive for systems is higher than that of Hamiltonian elliptic systems studied in \cite{ST18}. As shown in \cite{ST18}, the $C^{2,\alpha}$ regularity remains valid for systems with vanishing Neumann boundary condition, when $F=G=-\Delta$ and $p$ and $q$ are chosen in a certain way. 

Additionally, we prove a variant of a weak comparison principle, showing that if one can compare viscosity sub- and super-solutions on the boundary, then at least one of the functions that make the pair of these sub- and super-solutions can still be compared in the domain of the free boundary. The latter leads to several geometric measure estimates for the free boundary. Furthermore, an application of the regularity result implies that whenever $F$ and $G$ vanish at the origin, the free boundaries coincide, 
$$
\partial\{u+v>0\}=\partial\{u>0\}=\partial\{v>0\}.
$$
Another consequence is that if a non-negative entire solution of the system \eqref{1.1} vanishes at a point and has a certain growth at infinity, then it must be identically zero. Our results can be applied to systems with more than two equations.

The paper is organized as follows: in Section \ref{s2}, using Schaefer's fixed point argument, we prove existence of solutions for the system \eqref{1.1} (Proposition \ref{p2.1}) and derive a variant of a weak comparison principle (Lemma \ref{comparison}). In Section \ref{s3}, we prove regularity of solutions of the system \eqref{1.1} (see Theorem \ref{t3.1}) across the free boundary. Section \ref{s4} is dedicated to the study of geometric properties of the free boundary. In particular, we prove the non-degeneracy of viscosity solutions (Theorem \ref{t6.1}), the porosity of the free boundary (Lemma \ref{c6.1}), and coincidence of the free boundaries (Theorem \ref{fbcoincidence}). In Section \ref{s5}, two Liouville-type results are obtained for entire solutions (Theorem \ref{t5.1} and Theorem \ref{liouvthm}). We close the paper with an appendix on the analysis of radial solutions (Appendix \ref{appendixa}).
%%%%%%%%%%%%%%%%%%%%%%%%%%%%%%%%%%%%%
\section{Existence of solutions and weak comparison principle}\label{s2}
%%%%%%%%%%%%%%%%%%%%%%%%%%%%%%%%%%%%%
In this section, assuming that the boundary of the domain is of class $C^2$, we prove existence of viscosity solutions of the system 
\begin{equation}\label{systemdirichlet}
\left\{
\begin{array}{ccc}
F(D^2u,x)=v_+^{\,p}, & \text{in} & \Omega,\\[0.2cm]
G(D^2v,x)=u_+^q & \text{in} & \Omega,\\[0.2cm]
u=\varphi,\,v=\psi&\text{on}&\partial\Omega,
\end{array} 
\right. 
\end{equation}
where $\varphi$, $\psi\in C^{0,1}(\partial\Omega)$. Additionally, we prove a weak comparison principle. Throughout the paper, we assume that 
$$
p\ge0,\,\,\,q\ge0,\,\,\,pq<1,
$$
and $F$ and $G$ satisfy 
\begin{equation}\label{1.2}
	\lambda\|N\|\le\mathcal{F}(M+N,x)-\mathcal{F}(M,x)\le\Lambda\|N\|,\,\,\,\forall N\ge0,
\end{equation}
with $M\in\mathbb{S}^n$, $0<\lambda\le\Lambda$, and
\begin{equation}\label{uniformcontinuity}
	|\mathcal{F}(M,x)-\mathcal{F}(M,y)|\le\omega(|x-y|)\left(1+\|M\|\right),
\end{equation}
for a modulus of continuity $\omega$. Here $N$ is a non-negative definite symmetric matrix, therefore the norm $\|N\|$ is equal to the maximum eigenvalue of $N$. Solutions of \eqref{1.1} are understood in the viscosity sense according to the following definition.
\begin{definition}\label{d1.1}
	A couple of functions $u$, $v\in C({\Omega})$ is called a viscosity sub-solution of \eqref{1.1}, if whenever $\phi, \varrho\in C^2(\Omega)$ and $u-\phi$,  $v-\varrho$ attain local maximum at $x_0\in\Omega$, then
	\begin{equation*}
		\left\{
		\begin{array}{ccc}
			F\left(D^2\phi(x_0), x_0\right)\ge v^p_+(x_0), & \text{in} & \Omega,\\[0.2cm] 
			G\left(D^2\varrho(x_0), x_0\right)\ge u^q_+(x_0) & \text{in} & \Omega.
		\end{array} 
		\right.
	\end{equation*}
	Similarly, $u$, $v\in C({\Omega})$ is called a viscosity super-solution of \eqref{1.1}, if whenever $\phi,\varrho\in C^2(\Omega)$ and $u-\phi$, $v-\varrho$ attain local minimum at $x_0\in\Omega$, then
	\begin{equation*}
		\left\{
		\begin{array}{ccc}
			F\left(D^2\phi(x_0), x_0\right)\le v^p_+(x_0), & \text{in} & \Omega,\\[0.2cm] 
			G\left(D^2\varrho(x_0), x_0\right)\le u^q_+(x_0) & \text{in} & \Omega.
		\end{array} 
		\right.
	\end{equation*}
	A couple of functions $u$, $v\in C({\Omega})$ is called a viscosity solution of \eqref{1.1}, if it is both a viscosity sub- and super-solution of \eqref{1.1}.
\end{definition}
\subsection{Existence of solutions} As commented above, neither the classical Perron's method nor the standard comparison principle apply in this framework. Thus, we need new tools to tackle the issue. For that purpose, we suitably construct a continuous and compact map to use the following Schaefer's theorem (see, for example, \cite{Z85}), extending the idea from  \cite{CT20}. 
\begin{theorem}\label{schaefer}
	If $T:X \to X$, where $X$ is a Banach space, is continuous and compact, and the set
	$$
	\mathcal{E}=\{z \in X;\,\,\, \exists\,  \theta \in [0,1] \mbox{ such that } z=\theta T(z)  \}
	$$
	is bounded, then $T$ has a fixed point.
\end{theorem}
\begin{proposition}\label{p2.1}
	There exists a viscosity solution $(u,v)$ of system \eqref{systemdirichlet}. Moreover, $u$ and $v$ are globally Lipschitz functions.
\end{proposition}
\begin{proof}
	Let $f\in C^{0,1}(\overline{\Omega})$ and let $u\in C^{0,1}(\overline{\Omega})$ be the unique viscosity solution of the following problem:
	\begin{equation}\label{dirichlet1}
	\left\{
	\begin{array}{ccc}
	F(D^2u,x)=f_+^p(x), & \text{in} & \Omega,\\[0.2cm]
	u=\varphi & \text{on} & \partial\Omega,
	\end{array} 
	\right. 
	\end{equation}
	where $\varphi\in C^{0,1}(\partial\Omega)$. The existence and uniqueness of such $u$ are obtained by classical Perron's method (see \cite{CIL92,I89}). Similarly, for a function $g\in C^{0,1}(\overline{\Omega})$, there is a unique function $v\in C^{0,1}(\overline{\Omega})$ that solves the problem
	\begin{equation}\label{dirichlet2}
	\left\{
	\begin{array}{ccc}
	G(D^2v,x)=g_+^q(x), & \text{in} & \Omega,\\[0.2cm]
	v=\psi & \text{on} & \partial\Omega,
	\end{array} 
	\right. 
	\end{equation}
	where $\psi\in C^{0,1}(\partial\Omega)$. For a given pair of functions $(f,g)\in C^{0,1}(\overline{\Omega})\times C^{0,1}(\overline{\Omega})$, we then define 
	$$
	T : C^{0,1}(\overline{\Omega}) \times C^{0,1}(\overline{\Omega}) \to C^{0,1}(\overline{\Omega}) \times C^{0,1}(\overline{\Omega})  
	$$
	by $T(f,g):=(v,u)$, where $v$ and $u$ are the solutions of \eqref{dirichlet2} and \eqref{dirichlet1} respectively. Note that $T$ is well defined, as \eqref{dirichlet1} and \eqref{dirichlet2} have unique, globally Lipschitz solutions (see, for instance, \cite[Theorem 1.4]{SS14}). The aim is to use Schaefer's theorem, so we need to check that $T$ is continuous and compact. 
	
	Let $(f_k,g_k) \to (f,g)$ in $C^{0,1}(\overline{\Omega}) \times C^{0,1}(\overline{\Omega})$, and let 
	$$
	(v_k,u_k):=T(f_k,g_k).
	$$
	We want to show that $T(f_k,g_k) \to T(f,g)$ in $C^{0,1}(\overline{\Omega}) \times C^{0,1}(\overline{\Omega})$. By global Lipschitz regularity, for a universal constant $C>0$, depending on $\varphi$ and $\psi$, one has	
	$$	
	\|u_k\|_{C^{0,1}(\overline{\Omega})}\le C(\|(f_k)_+^p\|_\infty+\|\varphi\|_\infty),
	$$	
	and	 
	$$
	\|v_k\|_{C^{0,1}(\overline{\Omega})}\le C(\|(g_k)_+^q\|_\infty+\|\psi\|_\infty).
	$$
	Since $(f_k,g_k)$ converges, then it is bounded, therefore, $\|f_k\|_\infty$ and $\|g_k\|_\infty$ are bounded uniformly. Thus, by Arzel\'a-Ascoli theorem, the sequence $(v_k,u_k)$ is equicontinuous in $C^{0,1}(\overline{\Omega})\times C^{0,1}(\overline{\Omega})$ and hence, up to a sub-sequence, it converges to some $(v,u)$. Using stability of viscosity solutions under uniform limits, \cite[Proposition 2.1]{I89}, and passing to the limit, we obtain 
	$$
	T(f_k,g_k)=(v_k,u_k)\to(v,u)=T(f,g).
	$$
	The uniqueness of solutions of \eqref{dirichlet1} and \eqref{dirichlet2} guarantees that any other sub-sequence converges to $(v,u)$. Thus, $T$ is continuous.
	
	To see that $T$ is also compact, let $(f_k,g_k)$ be a bounded sequence in $C^{0,1}(\overline{\Omega}) \times C^{0,1}(\overline{\Omega})$. As above, $(v_k,u_k)=T(f_k,g_k) \in C^{0,1}(\overline{\Omega}) \times C^{0,1}(\overline{\Omega})$ is bounded. Then there is a convergent sub-sequence. 
	
	To use Theorem \ref{schaefer}, it remains to check that the set $\mathcal{E}$ of eigenvectors
	$$
	\mathcal{E}:=\{(f,g) \in C^{0,1}(\overline{\Omega}) \times C^{0,1}(\overline{\Omega});\,\,\, \exists\,  \theta \in [0,1] \mbox{ such that } (f,g)=\theta T(f,g) \}
	$$	
	is bounded. Observe that $(0,0)\in\mathcal{E}$ if and only if $\theta=0$. Hence, we can assume $\theta\neq0$. Let $(f,g)\in\mathcal{E}$. For any $0<\theta \leq 1$, we have	
	$$
	\left\{
	\begin{array}{ccc}
	F_\theta\left(D^2 g, x\right) = \theta f_+^p(x), & \text{in} & \Omega,\\[0.2cm]
	g=\theta\varphi & \text{on} & \partial\Omega,
	\end{array} 
	\right. 
	$$
	where $F_\theta(M,x):=\theta F (\theta^{-1} M,x)$ satisfies \eqref{1.2}, and
	$$
	\left\{
	\begin{array}{ccc}
	G_\theta\left(D^2 f, x\right) = \theta g_+^q(x), & \text{in} & \Omega,\\[0.2cm]
	f=\theta\psi & \text{on} & \partial\Omega,
	\end{array} 
	\right. 
	$$
	where $G_\theta(M,x):=\theta G (\theta^{-1} M,x)$ satisfies \eqref{1.2}. As above, from the Krylov-Safonov regularity theory, one has
	$$
	\|g\|_{C^{0,1}(\overline{\Omega})} \leq C(\|f_+^p\|_\infty+\|\varphi\|_\infty),
	$$ 
	and
	$$
	\|f\|_{C^{0,1}(\overline \Omega)}\leq C(\|g_+^q\|_\infty+\|\psi\|_\infty),
	$$ 
	where $C>0$ is a universal constant. However, since $F_\theta(D^2g,x)\ge0$, then $g$ takes its maximum on the boundary, and hence, is bounded by $\|\varphi\|_\infty$. 
	$$
	\|g\|_\infty\le\|\varphi\|_\infty.
	$$
	Similarly,
	$$
	\|f\|_\infty\le\|\psi\|_\infty.
	$$
	Thus, 
	$$
	\|g\|_{C^{0,1}(\overline{\Omega})} \leq C(\|\psi\|_\infty^p+\|\varphi\|_\infty)
	$$
	and
	$$
	\|f\|_{C^{0,1}(\overline{\Omega})} \leq C(\|\varphi\|_\infty^q+\|\psi\|_\infty),
	$$
	where $C>0$ is universal constant. This implies that $\mathcal{E}$ is bounded.
	
	Finally, we can now apply Theorem \ref{schaefer} to conclude that $T$ has a fixed point. This finishes the proof.
\end{proof}
\subsection{Weak comparison principle}
Next, we obtain a variant of a weak comparison principle for the system \eqref{1.1}. We assume that $F$ and $G$ are uniformly elliptic operators with constant coefficients, i.e., 
\begin{equation}\label{ellipticity}
	\lambda\|N\|\le\mathcal{F}(M+N)-\mathcal{F}(M)\le\Lambda\|N\|,\,\,\,\forall N\ge0,
\end{equation}
with $M\in\mathbb{S}^n$, $0<\lambda\le\Lambda$. As commented earlier, the absence of the classical Perron's method in our framework forces us to find an alternative way to derive information on relations between viscosity solutions and super-solutions of the system. Nevertheless, we discover a chain reaction and obtain a weak form of comparison principle for the system
\begin{equation}\label{constantcoeff}
	\left\{
	\begin{array}{ccc}
		F(D^2 u)&=&v_+^p,\,\,\,\text{ in }\,\,\,x\in\Omega,\\[0.2cm]
		G(D^2 v)&=&u_+^q,\,\,\,\text{ in }\,\,\,x\in\Omega.
	\end{array}
	\right.
\end{equation}
More precisely, although, in general, one cannot compare viscosity super- and sub- solutions $(u^*,v^*)$ and $(u_*,v_*)$ in the domain, we show that at least one of the following inequalities holds: 
\begin{equation*}
    u^*\ge u_*\,\,\textrm{ or }\,\,v^*\ge v_*,
\end{equation*}
provided both of them hold at the boundary, as shows the following lemma.
\begin{lemma}\label{comparison}
	Let $F$ and $G$ satisfy \eqref{ellipticity}. If $(u^*,v^*)$ is a viscosity super-solution and $(u_*,v_*)$ is a viscosity sub-solution of \eqref{constantcoeff} such that $u^*\ge u_*$ and $v^*\ge v_*$ on $\partial\Omega$, then
	$$
	\max\{u^*-u_*,v^*-v_*\}\ge0\,\,\textrm{ in }\,\,\Omega.
	$$    
\end{lemma}
\begin{proof}
	Suppose the conclusion of the lemma is not true, i.e., there exists $x_0\in\Omega$ such that
	\begin{equation}\label{contrad}
		u_*(x_0)>u^*(x_0)\,\,\,\textrm{ and }\,\,\,v_*(x_0)>v^*(x_0).
	\end{equation}
	Set
	\begin{equation}\nonumber
		m:=\sup\limits_{\overline{\Omega}} (u_*-u^*)>0,
	\end{equation}
	and for each $\varepsilon>0$ small, define
	$$
	m_\varepsilon := \sup\limits_{\overline{\Omega}\times \overline{\Omega}} \left(u_*(x)-u^*(y) - \frac{1}{2\varepsilon}|x-y|^2 \right) < \infty.
	$$
	Let $(x_\varepsilon,y_\varepsilon) \in \overline{\Omega} \times \overline{\Omega}$ be a point where the supremum is attained. Then, \cite[Lemma 3.1]{CIL92}, one has
	$$
	\lim\limits_{\varepsilon \to 0} \frac{1}{\varepsilon}|x_\varepsilon-y_\varepsilon|^2=0 \quad \mbox{and} \quad \lim\limits_{\varepsilon \to 0}m_\varepsilon = m.
	$$
	In particular, 
	$$
	\lim\limits_{\varepsilon \to 0}x_\varepsilon = \lim\limits_{\varepsilon \to 0} y_\varepsilon = z_0,
	$$
	where 
	$$
	u_*(z_0)-u^*(z_0)=m > 0 \geq \sup_{\,\partial\Omega} (u_*-u^*).
	$$
	Hence, for $\varepsilon$ sufficiently small, we have $x_\varepsilon, y_\varepsilon \in \Omega'$, for some $\Omega'\subset\subset\Omega$. Therefore (see \cite[Theorem 3.2]{CIL92}), there exist matrices $M,N \in \mathbb{S}^n$, such that
	\begin{equation}\label{jets}
		\left( \frac{x_\varepsilon-y_\varepsilon}{\varepsilon}, M \right) \in \overline{J}^{2,+}_\Omega u_*(x_\varepsilon) \quad \mbox{and} \quad \left( \frac{y_\varepsilon-x_\varepsilon}{\varepsilon}, N \right) \in \overline{J}^{2,-}_\Omega u^*(y_\varepsilon),
	\end{equation}
	where 
	\begin{equation}\nonumber
		-\frac{3}{\varepsilon}
		\left(
		\begin{array}{cc}
			I  & 0  \\
			0  & I
		\end{array}
		\right)
		\leq
		\left(
		\begin{array}{cc}
			M  & 0  \\
			0  & N
		\end{array}
		\right)    
		\leq
		\frac{3}{\varepsilon}
		\left(
		\begin{array}{cc}
			I  & -I  \\
			-I  & I
		\end{array}
		\right).
	\end{equation}
	In particular, this implies that $M \leq N$. Combining \eqref{ellipticity} and \eqref{jets}, we obtain 
	$$
	\left(v_*(x_\varepsilon)\right)_+^p \le F(M) \le F(N) \leq  \left(v^*(y_\varepsilon)\right)_+^p,
	$$
	for each $\varepsilon>0$. By letting $\varepsilon \to 0$, we conclude
	$$
	v^*(z_0) - v_*(z_0) \ge 0,
	$$
	which contradicts \eqref{contrad}. 
\end{proof}
%%%%%%%%%%%%%%%%%%%%%%%%%%%%%%%%%%%
\section{Regularity of solutions at the free boundary}\label{s3}
In this section, we prove our main result, deriving regularity for solutions of \eqref{1.1} across the free boundary, $\partial\{|(u,v)|>0\}$, where 
\begin{equation}\label{2.1}	
|(u,v)| := u_+^{\frac{1}{1+p}}+v_+^{\frac{1}{1+q}}.	
\end{equation}
As an auxiliary step, we show that bounded solutions of a system that vanish at a point can be flattened, meaning that they can be made smaller than a given positive constant once the right-hand side is suitably perturbed. As is seen below, there is a natural phenomenon, a \textit{domino effect}, when a certain control over one of the diffusions dictates a wave across the system. 

We say $(u,v)\ge 0$, if both $u$ and $v$ are non-negative. As our results are of a local nature, we consider $B_1$ instead of $\Omega$.
\begin{lemma}\label{l2.1}
	For every $\mu>0$, there exists $\kappa>0$ depending only on $\mu$, $p$, $q$, $\lambda$, $\Lambda$, $n$, such that if $u,\,v\in[0,1]$, $u(0)=v(0)=0$ and 
	\begin{equation*}
	\left\{
	\begin{array}{ccc}
	F(D^2u,x)&=&\delta v_+^p,\,\,\text{in}\,\,B_1,\\[0.2cm]
	G(D^2v,x)&=&u_+^q\,\,\text{in}\,\,B_1.
	\end{array} 
	\right.
	\end{equation*}
	in the viscosity sense for $\delta\in(0,\kappa)$, then
	$$
	\sup_{B_{1/2}}|(u,v)|\le\mu.
	$$
\end{lemma}
\begin{proof}
	We argue by contradiction and suppose the conclusion of the lemma is not true, i.e., for $\mu_0>0$, there exist sequences of functions $u_i$, $v_i$ such that $u_i,\,v_i\in[0,1]$, $u_i(0)=v_i(0)=0$ and 
	\begin{equation*}
	\left\{
	\begin{array}{ccc}
	F_i(D^2u_i,x)&=&i^{-1}\left(v_i\right)_+^p,\,\,\text{in}\,\,B_1,\\[0.2cm]
	G_i(D^2v_i,x)&=&\left(u_i\right)_+^q\,\,\text{in}\,\,B_1,
	\end{array} 
	\right.
	\end{equation*}
	where $F_i$, $G_i$ are elliptic operators satisfying \eqref{1.2} and \eqref{uniformcontinuity}, but
	\begin{equation}\label{2.2}
	\sup_{B_{1/2}}|(u_i,v_i)|>\mu_0.
	\end{equation}
	By the Krylov-Safonov regularity theory (see, for example, \cite{CC95}), up to sub-sequences, $u_i$ and $v_i$ converge locally uniformly in $B_{2/3}$ to respectively $u_\infty$  and $v_\infty$, as $i\rightarrow\infty$. Clearly, $u_\infty(0)=0$, $u_\infty\in[0,1]$, and by stability of viscosity solutions, $u_\infty$ in $B_{2/3}$ satisfies 
	$$
	F_\infty(D^2u_\infty,x)=0,
	$$
	where $F_\infty$ satisfies \eqref{1.2} and \eqref{uniformcontinuity}. The strong maximum principle then implies that $u_\infty\equiv0$. The latter provides that in addition to $v_\infty(0)=0$ and $v_\infty\in[0,1]$, one also has 
	$$
	G_\infty(D^2v_\infty,x)=0,
	$$ 
	where $G_\infty$ satisfies \eqref{1.2} and \eqref{uniformcontinuity}. This \textit{chain reaction} allows one to apply once more the strong maximum principle and conclude that $v_\infty\equiv0$, which contradicts \eqref{2.2}.
\end{proof}
Next, we apply Lemma \ref{l2.1} to obtain regularity of solutions. Geometrically, it reveals that viscosity solutions of the system \eqref{1.1} touch the free boundary in a smooth fashion. In fact, it provides quantitative information on the speed at which bounded viscosity solution detaches from the dead core. In the next section, we show that it is the exact speed at which viscosity solutions of \eqref{1.1} grow.
\begin{theorem}\label{t3.1}
	Let $(u,v)$ be a non-negative viscosity solution of \eqref{1.1} in $B_1$. There exists a constant $C>0$, depending only on $\lambda,\Lambda,p,q$, $\|u\|_\infty,\|v\|_\infty$, such that for  $x_0\in \partial\{|(u,v)|>0\} \cap B_{1/2}$ there holds
	\begin{equation}\label{mainest}
	|(u(x),v(x))|\le C\,|x-x_0|^{\frac{2}{1-pq}},
	\end{equation}
	for any $x \in B_{1/4}$. In particular,
	$$
	u(x)\le C\,|x-x_0|^{\frac{2(1+p)}{1-pq}} 
	\quad \mbox{and} \quad
	v(x) \le C\,|x-x_0|^{\frac{2(1+q)}{1-pq}}.
	$$
\end{theorem}

\begin{proof}
	Without loss of generality, we assume that $x_0=0$. 
	
	Take now $\mu=2^{-\frac{2}{1-pq}}$, and let $\kappa>0$ be as in Lemma \ref{l2.1}. We need to apply Lemma \ref{l2.1} to suitably chosen sequences of functions. We construct the first pair of functions in the following way. Set
	$$
	u_1(x):=b\cdot u(ax)\,\,\,\,\textrm{ and }\,\,\,\,v_1(x):=b\cdot v(ax),
	$$
	where 
	$$
	b:=\min\left\{1,\kappa^{\frac{1}{1-p}},\|u\|_\infty^{-1},\|v\|_\infty^{-1}\right\}\,\,\textrm{ and }\,\,a:=b^{\frac{q-1}{2}}. 
	$$
	Note that $u_1$, $v_1\in[0,1]$. As $0\in\partial\{|(u,v)|>0\}$, then $u_1(0)=v_1(0)=0$. Also $(u_1,v_1)\ge0$ is a viscosity solution of the system 
	\begin{equation}\label{3.2}
	\left\{
	\begin{array}{cclcc}
	\tilde{F}(D^2u_1,x)&=&b^{q-p}(v_{1}(x))_+^p & \text{in} & B_1,\\[0.2cm]
	\tilde{G}(D^2v_1,x)&=&(u_{1}(x))_+^q & \text{in} & B_1,
	\end{array} 
	\right.
	\end{equation}
	where $\tilde{F}(M,x):=b^qF(b^{-q}M,ax)$, $\tilde{G}(N,x):=b^qG(b^{-q}N,ax)$ are $(\lambda,\Lambda)$ - elliptic operators with uniformly continuous coefficients, i.e., satisfy \eqref{1.2} and \eqref{uniformcontinuity}. Applying Lemma \ref{l2.1}, we deduce
	\begin{equation}\label{3.3}
	\sup_{B_{1/2}}|(u_1,v_1)|\le 2^{-\frac{2}{1-pq}}.
	\end{equation}
	Next, we define sequences of functions $u_2,v_2:B_1\rightarrow\R$ by
	$$
	u_2(x):=2^{\frac{2}{1-pq}(1+p)}u_1\left(\frac{x}{2}\right)\,\,\,\textrm{ and }\,\,\,v_2(x):=2^{\frac{2}{1-pq}(1+q)}v_1\left(\frac{x}{2}\right).
	$$
	Using \eqref{3.3} we have that $u_2,v_2\in[0,1]$. Also $u_2(0)=v_2(0)=0$, and $(u_2,v_2)\geq0$ is a viscosity solution of \eqref{3.2}, for some uniformly elliptic operators $\tilde{F}$ and $\tilde{G}$. Hence, Lemma \ref{l2.1} provides with
	$$
	\sup_{B_{1/2}}|(u_2,v_2)|\le 2^{-\frac{2}{1-pq}}.
	$$
    Scaling back, we have
    $$
    \sup_{B_{1/4}}|(u_1,v_1)|\le 2^{-2\cdot\frac{2}{1-pq}}.
    $$	
    Same way we define $u_i,v_i:B_1\rightarrow\R$, $i=3,4,\ldots$, by
    $$
    u_i(x):=2^{\frac{2}{1-pq}(1+p)}u_{i-1}\left(\frac{x}{2}\right)\,\,\,\textrm{ and }\,\,\,v_i(x):=2^{\frac{2}{1-pq}(1+q)}v_{i-1}\left(\frac{x}{2}\right)
    $$
    and deduce that
    $$
    \sup_{B_{1/2}}|(u_i,v_i)|\le 2^{-\frac{2}{1-pq}}.
    $$
    Scaling back provides us with
    $$
    \sup_{B_{1/2^i}}|(u_1,v_1)|\le 2^{-i\cdot\frac{2}{1-pq}}.
    $$
    To finish the proof, for any given $\displaystyle r\in\left(0,\frac{a}{2}\right)$, we choose $i\in\mathbb{N}$ such that 
    $$
    2^{-(i+1)}<\frac{r}{a}\le2^{-i},
    $$
    and estimate
    \begin{equation*}
    \begin{array}{ccl}
    \displaystyle\sup_{B_r}|(u,v)|&\le&\displaystyle\sup_{B_{r/a}}|(u_1,v_1)|\\[0.5cm]
    &\le&\displaystyle\sup_{B_{1/2^i}}|(u_1,v_1)|\\[0.5cm]
    &\le&2^{-i\cdot\frac{2}{1-pq}}\\[0.3cm]
    &\le&\left(\frac{2}{a}\right)^{\frac{2}{1-pq}}r^{\frac{2}{1-pq}}\\[0.3cm]
    &=&Cr^{\frac{2}{1-pq}}.
    \end{array} 
    \end{equation*}    
\end{proof}

%%%%%%%%%%%%%%%%%%%%%%%%%%%%%%%%%%%%%
\section{Geometric properties of the free boundary}\label{s4}
%%%%%%%%%%%%%%%%%%%%%%%%%%%%%%%%%%%%%
In this section, we prove geometric measure estimates for the free boundary. In particular, we obtain non-degeneracy and porosity results, and conclude the coincidence of the free boundaries for the system.
\subsection{Non-degeneracy of solutions}
We obtain a non-degeneracy estimate for viscosity solutions with a careful analysis of radial super-solutions (see Appendix \ref{appendixa} below). As was established by Theorem \ref{t3.1}, the quantity $|(u,v)|$, defined by \eqref{2.1}, grows not faster than $r^{\frac{2}{1-pq}}$. Here we prove that it grows exactly with that rate (see Figure \ref{figure1}), provided the operators have constant coefficients. Using Lemma \ref{comparison} and Proposition \ref{radialsupersolution}, we obtain non-degeneracy of solutions, whenever
\begin{equation}\label{extracondition}
	F(0)=G(0)=0.
\end{equation}
\begin{theorem}\label{t6.1}
	If \eqref{ellipticity}, \eqref{extracondition} hold, $(u,v)$ is a non-negative bounded viscosity solution of \eqref{constantcoeff} in $B_1$, and $y\in\overline{\left\{|(u,v)|>0\right\}}\cap B_{1/2}$, then for a constant $c>0$, depending on the dimension, one has
	$$
	\sup_{\overline{B}_r(y)}|(u,v)|\ge cr^{\frac{2}{1-pq}},
	$$
	for any $r\in\left(0,\frac{1}{2}\right)$.
\end{theorem}
\begin{proof}
Since $u$ and $v$ are continuous, it is enough to prove the theorem for the points $y\in\left\{|(u,v)|>0\right\}\cap B_{1/2}$, i.e. $u(y)>0$ and $v(y)>0$. By translation, we may assume, without loss of generality, that $y=0$. Let now $\tilde{u}$ and $\tilde{v}$ be as in Proposition \ref{radialsupersolution}. Note that if for a $\xi\in\partial B_r(y)$ one has
	\begin{equation}\label{enoughtoshow}
		\left(u(\xi),v(\xi)\right)\ge\left(\tilde{u}(\xi),\tilde{v}(\xi)\right),
	\end{equation} 
	then
	$$
	\sup_{\overline{B}_r}|(u,v)|\ge|(u(\xi),v(\xi))|\ge cr^{\frac{2}{1-pq}}.
	$$
	Hence, it is enough to check \eqref{enoughtoshow}. Suppose it is not true, i.e., 
	$$
	(u(\xi),v(\xi))<(\tilde{u}(\xi),\tilde{v}(\xi)),\,\,\,\forall\xi\in\partial B_r.
	$$
	Define
	\begin{equation*}
	\hat{u}=
	\left\{
	\begin{array}{ccc}
	\min\{u,\tilde{u}\} & \text{in} & B_r\\[0.2cm]
	u & \text{in} & B_r^c
	\end{array} 
	\right.
	\end{equation*}	
	and
	\begin{equation*}
	\hat{v}=
	\left\{
	\begin{array}{ccc}
	\min\{v,\tilde{v}\} & \text{in} & B_r\\[0.2cm]
	v & \text{in} & B_r^c.
	\end{array} 
	\right.
	\end{equation*}	
	Since $(\tilde{u},\tilde{v})$ is a super-solution and $(u,v)$ is a solution of \eqref{constantcoeff}, then $(\hat{u},\hat{v})$ also is a super-solution (see \cite[Lemma 3.1]{IK912}). On the other hand, Lemma \ref{comparison} implies
	$$
	0<u(0)\le\hat{u}(0)=0\,\,\,\textrm{ or }\,\,\,0<v(0)\le\hat{v}(0)=0,
	$$
	which is a contradiction.
\end{proof}
\begin{corollary}\label{c4.1}
	If \eqref{ellipticity}, \eqref{extracondition} hold, $(u,v)$ is a non-negative viscosity solution of \eqref{constantcoeff}, and $x_0\in\partial\{|(u,v)|>0\}$, then
	$$
	c\,r^{\frac{2}{1-pq}}\le\sup_{B_r(x_0)}|(u,v)|\le C\,r^{\frac{2}{1-pq}} 
	$$
	where $r>0$ is small enough and $c,\,C>0$ are constants depending on the dimension.
\end{corollary}
\begin{figure}
	\psscalebox{0.5 0.5} % Change this value to rescale the drawing.
	{
		\begin{pspicture}(0,-4.646041)(19.06,4.646041)
		\definecolor{colour0}{rgb}{0.2627451,0.25490198,0.25490198}
		\definecolor{colour1}{rgb}{0.28627452,0.2784314,0.2784314}
		\definecolor{colour2}{rgb}{0.02745098,0.47058824,0.5137255}
		\definecolor{colour3}{rgb}{0.53333336,0.53333336,0.53333336}
		\psline[linecolor=black, linewidth=0.04](0.0,-3.7611387)(13.2,-3.7611387)
		\psbezier[linecolor=colour0, linewidth=0.04](4.8,-3.7611387)(7.2,-3.7611387)(11.2,-0.56113863)(13.2,4.638861381081912)
		\psbezier[linecolor=colour1, linewidth=0.04](5.2,-3.7611387)(7.6,-3.7611387)(11.2,-2.5611386)(13.2,-0.16113861891808823)
		\psbezier[linecolor=colour2, linewidth=0.06](5.2,-3.7611387)(6.8,-3.7611387)(7.0938168,-2.984024)(8.0,-2.561138618918088)(8.906183,-2.1382532)(10.865196,-2.0394187)(11.6,-1.3611386)(12.334804,-0.6828585)(12.4,1.4388614)(13.2,2.2388613)
		\rput[bl](1.0,-3.5){{\huge dead-core}}
		\rput[bl](8.7,1.5){{\huge $|x|^\frac{2}{1-pq} \sim$}}
		\rput[bl](11.5,-2.9){{\huge $\sim |x|^\frac{2}{1-pq}$}}
		\rput[bl](13.3,1.2){{\huge $|(u,v)|$}}
		\psline[linecolor=colour3, linewidth=0.06](0.0,-3.7611387)(5.2,-3.7611387)
		\psdots[linecolor=black, dotsize=0.2](5.2,-3.7611387)
		\rput[bl](12.3,-4.5611386){{\huge $\mathbb{R}^n$}}
		\end{pspicture}
	}
	\caption{The growth of $|(u,v)|$ near the free boundary}\label{figure1}
\end{figure}

\subsection{Porosity of the free boundary}
Next, we establish positive density and porosity results for the free boundary. Moreover, we show that $(n-\varepsilon)$-dimensional Hausdorff measure of the free boundary is finite, where $\varepsilon>0$. We use $|E|$ for the $n$-dimensional Lebesgue measure of the set $E$. We recall the definition of a porous set.
\begin{definition}\label{porous}
	The set $E\subset\R^n$ is called porous with porosity $\sigma$, if there is $R>0$ such that $\forall x\in E$ and $\forall r\in (0,R)$ there exists $y\in\R^n$ such that
	$$
	B_{\sigma r}(y)\subset	B_{r}(x)\setminus E.
	$$
\end{definition}
A porous set of porosity $\sigma$ has Hausdorff dimension not exceeding
$n-c\sigma^n$, where $c>0$ is a constant depending only on the dimension. In particular, a porous set has Lebesgue measure zero (see \cite[Theorem 2.1]{KR97}, for instance).
\begin{lemma}\label{c6.1}
	If \eqref{ellipticity}, \eqref{extracondition} hold, $(u,v)$ is a non-negative bounded viscosity solution of \eqref{constantcoeff} in $B_1$, and $y\in\partial\{|(u,v)|>0\}\cap B_{1/2}$, then
	$$
	|B_\rho(y)\cap\{|(u,v)|>0\}|\ge c\rho^n,\,\,\,\,\forall\rho\in\left(0,\frac{1}{2}\right),
	$$
	where $c>0$ is a constant depending only on $\lambda$, $\Lambda$, $\|u\|_\infty$, $\|v\|_\infty$, $p$, $q$ and $n$. Moreover, the free boundary is a porous set, and as a consequence
	$$
	\mathcal{H}^{n-\varepsilon}\left(\partial\{|(u,v)|>0\}\cap B_{1/2}\right)<\infty,
	$$
	where $\varepsilon>0$ is a constant depending only on $\lambda$, $\Lambda$, $p$, $q$ and $n$.
\end{lemma}
\begin{proof}
	For a $\rho\in(0,1/2)$, Theorem \ref{t6.1} guarantees the existence of a point $\xi_\rho$ such that
	\begin{equation}\label{positivedensity}
	|\left(u(\xi_\rho),v(\xi_\rho)\right)|\ge c\rho^{\frac{2}{1-pq}}.
	\end{equation}
	On the other hand, for a $\tau>0$, take
	$$
	y_0\in B_{\tau\rho}(\xi_\rho)\cap\partial\{|(u,v)|>0\}\neq\emptyset.
	$$
	Recalling Theorem \ref{t3.1} and using \eqref{positivedensity}, we have
	$$
	c\rho^{\frac{2}{1-pq}}\le|(u(\xi_\rho),v(\xi_\rho))|\le\sup_{B_{\rho\tau}(y_0)}\le C\left(\rho\tau\right)^{\frac{2}{1-pq}},
	$$
	which is a contradiction once
	$$
	\tau<\left(\frac{c}{C}\right)^\frac{1-pq}{2}.
	$$
	Therefore, for $\tau>0$ small enough
	$$
	B_{\tau\rho}(\xi_\rho)\subset\{|(u,v)|>0\},
	$$
	and hence
	$$
	|B_\rho(y)\cap\{|(u,v)|>0\}|\ge|B_\rho(y)\cap B_{\tau\rho}(\xi_\rho)|\ge c\rho^n.
	$$
	It remains to check the finiteness of the $(n-\varepsilon)$-dimensional Hausdorff measure of the free boundary, which, as observed above, is a consequence of its porosity. To show the latter, it is enough to take 
	$$
	y^*:=t\xi_\rho+(1-t)y,
	$$ 
	where $t$ is close enough to $1$ to guarantee
	$$
	B_{\frac{\tau}{2}\rho}(y^*)\subset B_\tau(\xi_\rho)\cap B_\rho(y)\subset B_\rho(y)\setminus\partial\{|(u,v)|>0\},
	$$ 
	i.e., the set $\partial\{|(u,v)|>0\}\cap B_{1/2}$ is a $\frac{\tau}{2}$-porous set, and the result follows.
\end{proof}
\subsection{Coincidence of the free boundaries}
We finish this section by observing that the free boundary of \eqref{constantcoeff} coincides with that of $u$ and $v$ whenever both $p$ and $q$ are strictly positive, and $F$ and $G$ are convex or concave functions (for an illustration, see Figure \ref{figure_radial}).
\begin{theorem}\label{fbcoincidence}
If \eqref{ellipticity}, \eqref{extracondition} hold, $F$ and $G$ are convex (concave) functions, $(u,v)$ is a non-negative viscosity solution of \eqref{constantcoeff} and $pq>0$, then 
	$$
	\partial\{|(u,v)|>0\}=\partial\{u>0\}=\partial\{v>0\}.
	$$
\end{theorem}

\begin{proof}
Since $F$ is convex (concave), and the right-hand side is Lipschitz, then $u\in C_{\loc}^{2,\mu}$, for some $\mu\in(0,1)$, \cite{CC95}. Observe that $|D^2u|=0$ on $\partial\{u>0\}$. Indeed, for any fixed indexes $i$, $j$ and for $z\in\partial\{u>0\}$, Theorem \ref{t3.1} implies
\begin{equation*}
\begin{split}
    |D_{ij}u(z)|&=\lim_{t\to 0}\left|\dfrac{u\left(z+t(e_i+e_j)\right)-u(z+te_i)-u(z+te_j)+u(z)}{t^2}\right|\\
    &=\lim_{t\to 0}\left|\dfrac{u\left(z+t(e_i+e_j)\right)-u(z+te_i)-u(z+te_j)}{t^2}\right|\\
    &\leq C\lim_{t\to0}|t|^{\frac{2(1+p)}{1-pq}-2}=0,
\end{split}
\end{equation*}
where in the last step we used 
$$
\frac{2(1+p)}{1-pq}>2.
$$
Here $e_l$, $l=1,2,\ldots,n$, is the $l$-th vector of the standard orthonormal basis in $\R^n$. Thus, $D^2u(z)=0$, and \eqref{extracondition} gives $F\left(D^2u(z)\right)=F(0)=0$. Therefore, from the first equation in \eqref{constantcoeff}, we conclude that $v_+(z)=0$, i.e.,
	\begin{equation}\label{7.1}
		z\in\{v=0\}.
	\end{equation}
	On the other hand, if $B_\rho\subset\{v=0\}$, then $D^2v=0$ in $B_\rho$, which yields $u=0$ in $B_\rho$. Thus, $\{v=0\}^\mathrm{o}=\{u=0\}^\mathrm{o}$, which combined with \eqref{7.1}, implies that $z\in\partial\{v>0\}$, i.e., $\partial\{u>0\}\subset\partial\{v>0\}$. 
 
 Similarly, we can see that $\partial\{v>0\}\subset\partial\{u>0\}$, and the proof follows.
\end{proof}

%%%%%%%%%%%%%%%%%%%%%%%%%%%%%%%%%%%%%
\section{Liouville type results}\label{s5}
%%%%%%%%%%%%%%%%%%%%%%%%%%%%%%%%%%%%%
As an application of the regularity result above, exploiting the ideas from \cite{TU14} (see also \cite{AT22}), we obtain a Liouville type theorem for solutions of \eqref{1.1} in $\Omega=\R^n$. We refer to them as entire solutions of the system. Although Theorem \ref{t3.1} provides regularity information only across the free boundary, it is enough to show that the only entire solution, which vanishes at a point and has a growth suitably controlled at infinity, is the trivial one.
\begin{theorem}\label{t5.1}
		Let $(u,v)$ be a non-negative viscosity solution of
\begin{equation}\label{entire}
		\left\{
		\begin{array}{ccc}
		F(D^2u,x)=v_+^p & \text{in} & \mathbb{R}^n\\[0.2cm]
		G(D^2v,x)=u_+^q & \text{in} & \mathbb{R}^n,
		\end{array} 
		\right.
\end{equation}			
and $u(x_0)=v(x_0)=0$. If
	\begin{equation}\label{5.1}
		|(u(x),v(x))|=o\left(|x|^\frac{2}{1-pq}\right),\,\,\,\textrm{ as }\,\,\,|x|\rightarrow\infty,
	\end{equation}
	then $u\equiv v\equiv0$.
\end{theorem}
\begin{proof}
Without loss of generality, we may assume that $x_0=0$. We then define
$$
u_k(x):=k^{\frac{-2(1+p)}{1-pq}}u(kx)\,\,\,\textrm{ and }\,\,\,v_k(x):=k^{\frac{-2(1+q)}{1-pq}}v(kx),
$$
for $k\in\mathbb{N}$ and note that the pair $(u_k,v_k)$ is a viscosity solution of the system \eqref{entire} in $B_1$. Moreover, $u_k(0)=v_k(0)=0$, since $0\in\partial\{|(u,v)|>0\}$. Therefore, one can apply Theorem \ref{t3.1} to estimate $|(u_k,v_k)|$. More precisely, let $x_k\in\overline{B}_r$ be such that $|(u_k,v_k)|$ reaches its supremum at that point, for $r>0$ small. Applying Theorem \ref{t3.1}, we obtain
\begin{equation}\label{5.2}
	|(u_k(x_k),v_k(x_k))|\le C_k|x_k|^{\frac{2}{1-pq}},
\end{equation}
where $C_k>0$ goes to zero, as $k\rightarrow0$. Thus, if $|kx_k|$ is bounded as $k\rightarrow\infty$, then $|(u(kx_k),v(kx_k))|$ remains bounded. The latter implies
\begin{equation}\label{5.3}
	|(u_k,v_k)|\rightarrow0,\,\,\,\textrm{ as }\,\,\,k\rightarrow\infty.
\end{equation}
Note that due to \eqref{5.1}, \eqref{5.3} remains true also in the case when $|kx_k|\rightarrow\infty$, as $k\rightarrow\infty$. In fact, from \eqref{5.1} one gets
$$
|(u_k(x_k),v_k(x_k))|\le|kx_k|^{-\frac{2}{1-pq}}k^{-\frac{2}{1-pq}}\rightarrow0.
$$
Our aim is to show that both $u$ and $v$ are identically zero. Let us assume this is not the case. If $y\in\R^n$ is such that $|(u(y),v(y))|>0$, then by choosing $k\in\mathbb{N}$ large enough so $y\in B_{kr}$, and using \eqref{5.2}, \eqref{5.3}, we estimate
\begin{equation*}
\begin{array}{ccl}
\displaystyle\frac{|(u(y),v(y))|}{|y|^{\frac{2}{1-pq}}}&\le&\displaystyle\sup_{B_{kr}}\frac{|(u(x),v(x))|}{|x|^{\frac{2}{1-pq}}}\\[0.8cm]
&=&\displaystyle\sup_{B_r}\frac{|(u_k(x),v_k(x))|}{|x|^{\frac{2}{1-pq}}}\\[0.8cm]
&\le&\displaystyle\frac{|(u(y),v(y))|}{2|y|^{\frac{2}{1-pq}}},
\end{array} 
\end{equation*}    
which implies $|(u(y),v(y))|=0$, a contradiction.
\end{proof}
Additionally, if $F$ and $G$ have constant coefficients and both of them vanish at the origin, then Theorem \ref{t5.1} can be improved by relaxing \eqref{5.1} and not requiring $\partial \{|(u,v)>0|\} \neq \emptyset$.
\begin{theorem}\label{liouvthm}
	Let \eqref{ellipticity}, \eqref{extracondition} hold. If $(u,v)$ is a non-negative viscosity sub-solution of 
	\begin{equation}\label{entire1}
		\left\{
		\begin{array}{ccc}
			F(D^2u)=v_+^p & \text{in} & \mathbb{R}^n\\[0.2cm]
			G(D^2v)=u_+^q & \text{in} & \mathbb{R}^n,
		\end{array} 
		\right.
	\end{equation}		
	and 
	\begin{equation}\label{asymptotic}
	\limsup\limits_{|x| \to \infty} \frac{|(u(x),v(x))|}{|x|^{\frac{2}{1-pq}}} <\min\{A^{\frac{1}{1+p}},B^{\frac{1}{1+q}}\},
	\end{equation}
	where $A$ and $B$ are the constants defined by \eqref{constantA} and \eqref{constantB} respectively, then $u\equiv v\equiv 0$.
\end{theorem}
\begin{proof}
	Set
	$$
	\mathcal{S}_R:=\sup_{\partial B_R} |(u,v)| = \sup_{\partial B_R}\left( u^{\frac{1}{1+p}}+v^{\frac{1}{1+q}}\right).
	$$
    Using \eqref{asymptotic}, we choose $\theta<1$ and $R \gg 1$ such that
	\begin{equation}\label{asymptotic2}
	R^{-\frac{2}{1-pq}}\mathcal{S}_R  \leq \theta m,
	\end{equation}
	where $m:=\min\{A^{\frac{1}{1+p}},B^{\frac{1}{1+q}}\}$. Recall that the pair of functions	
	$$
	\tilde u(x):= A(|x|-r)_+^{\frac{2(1+p)}{1-pq}}\,\,\,\textrm{ and }\,\,\,\tilde v(x):= B(|x|-r)_+^{\frac{2(1+q)}{1-pq}},\,\,\,x \in\R^n,
	$$
	where $r>0$, is a viscosity super-solution of \eqref{entire1} (Proposition \ref{radialsupersolution}), and the corresponding dead-core is the ball of radius $r$ centered at $0$, that is, $\{|(\tilde{u},\tilde{v})|=0\}=B_r$. For $R\gg 1$, taking
	\begin{equation}\label{radius}
	r:=R-\left[\frac{1}{m}\mathcal{S}_R\right]^{\frac{1-pq}{2}} \geq (1-\theta^{\frac{1-pq}{2}}) R,
	\end{equation}
	from \eqref{asymptotic2} on $\partial B_R$ we have
	$$
	\tilde u = A(R-r)^{\frac{2(1+p)}{1-pq}}_+ = A\left[\frac{1}{m}\mathcal{S}_R\right]^{1+p}\ge\sup\limits_{\partial B_R} u
	$$
	and
	$$
	 \tilde v = B(R-r)^{\frac{2(1+q)}{1-pq}}_+ = B\left[\frac{1}{m}\mathcal{S}_R\right]^{1+q}\ge\sup\limits_{\partial B_R} v.
	$$
	Then, by Lemma \ref{comparison}, we deduce
	$$
	\max\{\tilde u - u, \tilde v - v\} \geq 0\,\,\,\mbox{ in } B_R.
	$$
	Thus, for a fixed point $x\in\R^n$, choosing $R$ large enough and using \eqref{radius}, we conclude that either $u(x)=0$ or $v(x)=0$, i.e., 
    $$
	u(x)v(x)=0,\,\,\,x\in\R^n.
    $$
The latter implies 
     \begin{equation}\label{uv}
     	\{u>0\}\subset\{v=0\}\,\,\,\textrm{ and }\,\,\,\{v>0\}\subset\{u=0\}.
     \end{equation}
	Suppose there exists $y\in\R^n$ such that $u(y)>0$. Since $u$ is continuous, then it remains positive in a neighborhood of $y$. From \eqref{uv} we obtain $v=0$ in that neighborhood. On the other hand, from \eqref{extracondition} we have $0=G(0)\ge u^q_+$, which implies that $u(y)=0$, a contradiction.
\end{proof}
\begin{remark}\label{r5.1}
	Observe that the constant on the right-hand side of \eqref{asymptotic} is sharp. Indeed, for the pair of functions $(\tilde{u},\tilde{v})$ defined above, one has equality in \eqref{asymptotic}, but neither $\tilde{u}$ nor $\tilde{v}$ is identically zero.
\end{remark}

\medskip

\textbf{Acknowledgments.} DJA is partially supported by CNPq and grant 2019/0014 Paraiba State Research Foundation (FAPESQ). He thanks the Centre for Mathematics of the University of Coimbra (CMUC) and the Abdus Salam International Centre for Theoretical Physics (ICTP) for great hospitality during his research visits. RT was partially supported by the King Abdullah University of Science and Technology (KAUST), by the Centre for Mathematics of the University of Coimbra (funded by the Portuguese Government through FCT/MCTES, DOI 10.54499/UIDB/00324/2020), and by FCT, DOI 10.54499/2022.02357.CEECIND/CP1714/CT0001. RT thanks the Department of Mathematics of the Universidade Federal da Para\'iba (UFPB) for hospitality and great working conditions during his research visit. 

\appendix

\section{Radial analysis}\label{appendixa}
Here we construct a radial viscosity super-solution for the system \eqref{constantcoeff}. Understanding these radial solutions of the system plays a crucial role in proving the non-degeneracy of solutions and certain measure estimates. As was shown in Section \ref{s2}, the system \eqref{constantcoeff} has a solution. In certain cases, solutions of the system \eqref{constantcoeff} can be constructed explicitly, as shown in the following example. One can easily see that the pair of functions 
$$
u(x):=A\left(|x|-\rho\right)_+^\frac{2(1+p)}{1-pq} \quad \mbox{and} \quad 
v(x):=B\left(|x|-\rho\right)_+^\frac{2(1+q)}{1-pq},
$$
where $A$, $B$ are universal constants depending only on $p$, $q$ and $n$ and $\rho\ge0$, is a viscosity solution for the system (see Figure \ref{figure_radial})
\begin{equation*}
	\left\{
	\begin{array}{ccc}
		\Delta u=v_+^{\,p}, & \text{in} & \R^n,\\[0.2cm]
		\Delta v=u_+^q & \text{in} & \R^n.
	\end{array} 
	\right. 
\end{equation*}
In general, as the operators $F$ and $G$ are assumed to be uniformly elliptic, one can show that the system \eqref{constantcoeff} has a radial viscosity sub- and super-solutions, provided \eqref{extracondition} holds. Below we construct such functions. For that purpose, let us define
$$
\tilde{u}(x):=A|x|^\alpha\,\,\,\,\,\,\textrm{ and }\,\,\,\,\,\tilde{v}(x):=B|x|^\beta,
$$
where 
$$
\alpha:=\frac{2(1+p)}{1-pq},\,\,\,\beta:=\frac{2(1+q)}{1-pq}
$$
and $A$, $B$ are positive constants to be chosen a posteriori. Direct computation gives
$$
\tilde{u}_{ij}(x)=A\alpha\left[(\alpha-2)|x|^{\alpha-4}x_ix_j+\delta_{ij}|x|^{\alpha-2}\right],
$$
where $\delta_{ij}=1$, if $i=j$ and is zero otherwise. Thus, using the ellipticity of $F$ and the fact that $F(0)=0$ in $\R^n$, we estimate
\begin{equation}\label{radial1}
	F(D^2\tilde{u})\le\Lambda\left[A\alpha(\alpha-1)+(n-1)A\alpha\right]|x|^{\alpha-2}.
\end{equation}
The aim is to choose the constants $A$ and $B$ such that $(\tilde{u},\tilde{v})$ is a viscosity super-solution of the system \eqref{constantcoeff}. If $B$ is such that
\begin{equation}\label{B}
	\Lambda\left[A\alpha(\alpha-1)+(n-1)A\alpha\right]=B^p,
\end{equation}
then \eqref{radial1} yields
$$
F(D^2\tilde{u})\le \left[B|x|^{\frac{2(1+q)}{1-pq}}\right]^p.
$$
Similarly,
\begin{equation}\label{radial2}
	G(D^2\tilde{v})\le\Lambda\left[B\beta(\beta-1)+(n-1)B\beta\right]|x|^{\beta-2},
\end{equation}
so if $A$ is such that
\begin{equation}\label{A}
	\Lambda\left[B\beta(\beta-1)+(n-1)B\beta\right]=A^q,
\end{equation}
then \eqref{radial2} gives
$$
G(D^2\tilde{v})\le \left[A|x|^{\frac{2(1+p)}{1-pq}}\right]^q.
$$
Thus, the constants $A$ and $B$ are chosen to satisfy \eqref{B} and \eqref{A}, which provides
\begin{equation}\label{constantA}
	A=\Lambda^{\frac{p+1}{pq-1}}\left[\alpha(\alpha-1)+\alpha(n-1)\right]^\frac{1}{pq-1}\left[\beta(\beta-1)+\beta(n-1)\right]^{\frac{p}{pq-1}}
\end{equation}
and
\begin{equation}\label{constantB}
	B=\Lambda^{\frac{q+1}{pq-1}}\left[\beta(\beta-1)+\beta(n-1)\right]^\frac{1}{pq-1}\left[\alpha(\alpha-1)+\alpha(n-1)\right]^{\frac{q}{pq-1}}.
\end{equation}
In conclusion, $(\tilde{u},\tilde{v})$ is a radial viscosity super-solution of the system \eqref{constantcoeff}. We state it below for reference.
\begin{figure}
	\psscalebox{0.5 0.5} % Change this value to rescale the drawing.
	{ 
		\begin{pspicture}(0,-4.7624373)(20.1,4.7624373)
			\definecolor{colour0}{rgb}{0.019607844,0.03529412,0.03529412}
			\definecolor{colour1}{rgb}{0.84705883,0.827451,0.827451}
			\definecolor{colour2}{rgb}{0.007843138,0.4627451,0.5411765}
			\definecolor{colour3}{rgb}{0.29411766,0.27450982,0.27450982}
			\definecolor{colour4}{rgb}{0.36078432,0.3529412,0.3529412}
			\definecolor{colour5}{rgb}{0.30980393,0.2901961,0.2901961}
			\definecolor{colour6}{rgb}{0.003921569,0.41568628,0.52156866}
			\definecolor{colour7}{rgb}{0.007843138,0.44313726,0.5254902}
			\psellipse[linecolor=colour0, linewidth=0.02, fillstyle=solid,fillcolor=colour1, dimen=outer](10.1,-3.637535)(2.2,0.6)
			\psellipse[linecolor=colour2, linewidth=0.04, linestyle=dashed, dash=0.17638889cm 0.10583334cm, dimen=outer](10.08,4.0424647)(4.56,0.72)
			\psellipse[linecolor=colour3, linewidth=0.04, linestyle=dashed, dash=0.17638889cm 0.10583334cm, dimen=outer](10.09,1.1624649)(4.97,0.8)
			\psbezier[linecolor=colour4, linewidth=0.04](5.12,1.1624649)(5.4266667,-0.37467805)(6.58,-3.2261066)(7.92,-3.637535193362155)
			\psbezier[linecolor=colour5, linewidth=0.04](15.06,1.1824648)(14.193334,-1.334678)(13.6,-3.2061067)(12.26,-3.617535193362155)
			\psbezier[linecolor=colour6, linewidth=0.04](5.56,3.9624648)(5.7466664,0.5653219)(6.56,-3.2261066)(7.9,-3.637535193362155)
			\psbezier[linecolor=colour7, linewidth=0.04](14.66,4.0224648)(14.473333,0.6253219)(13.66,-3.1661067)(12.32,-3.577535193362155)
			\psline[linecolor=black, linewidth=0.02, arrowsize=0.05291667cm 2.0,arrowlength=1.4,arrowinset=0.0]{->}(7.94,-4.237535)(8.573181,-2.5775352)(9.04,-3.0623934)
			\rput[bl](6.8,-4.677535){{\Large dead-core}}
			\psdots[linecolor=black, dotsize=0.1](10.1,-3.5975351)
			\psline[linecolor=black, linewidth=0.02, linestyle=dashed, dash=0.17638889cm 0.10583334cm](10.12,-3.617535)(12.26,-3.617535)
			\rput[bl](10.85,-3.5){{\Large $\rho$}}
			\rput[bl](14.8,-1.3){{\huge $\sim (|x|-\rho)_+^\alpha$}}
			\rput[bl](1.5,2.1){{\huge $ (|x|-\rho)_+^\beta \sim$}}
		\end{pspicture}
	}
	\caption{Radial solution}\label{figure_radial}
\end{figure}
\begin{proposition}\label{radialsupersolution}
	If $F$, $G$ satisfy \eqref{ellipticity} and \eqref{extracondition}, then the pair of functions 
	$$
	\tilde{u}(x):=A|x|^\frac{2(1+p)}{1-pq}\,\,\,\,\,\,\textrm{ and }\,\,\,\,\,\tilde{v}(x):=B|x|^\frac{2(1+q)}{1-pq},
	$$
	where $A$ and $B$ are defined by \eqref{constantA} and \eqref{constantB} respectively, is a viscosity super-solution of the system \eqref{constantcoeff}. Similarly, $(\tilde{u},\tilde{v})$ is a radial viscosity sub-solution of \eqref{constantcoeff}, if in the definition of $A$, $B$ the constant $\Lambda$ is substituted by $\lambda$.
\end{proposition}


\begin{thebibliography}{99}	
    \bibitem{AP86} H.W. Alt and D. Phillips, \textit{A free boundary problem for semilinear elliptic equations}, J. Reine Angew. Math. 368 (1986), 63--107.
	\bibitem{ALT16} D.J. Ara\'ujo, R. Leit\~ao and E. Teixeira, \textit{Infinity Laplacian equation with strong absorptions}, J. Functional Analysis 270 (2016), 2249--2267.
	\bibitem{AT22} D.J. Ara\'ujo and R. Teymurazyan, \textit{An optimal Liouville theorem for the porous medium equation}, Arch. Math. (Basel) 118 (2022), 427--433.    
	\bibitem{CC95} L.A. Caffarelli and X. Cabr\'e, \textit{Fully nonlinear elliptic equations}, Colloquium Publications 43 (1995), 104 pages.
	\bibitem{CDV18} L. Caffarelli, L. Duque and H. Vivas, \textit{The two membranes problem for fully nonlinear operators}, Discrete Contin. Dyn. Syst. 38 (2018), 6015--6027.
	\bibitem{CDS17} L. Caffarelli, D. De Silva and O. Savin, \textit{The two membranes problem for different operators}, Ann. Inst. H. Poincar\'e C Anal. Non Lin\'eaire 34 (2017), 899--932.
	\bibitem{CT20} L.A. Caffarelli and I. Tomasetti, \textit{Fully nonlinear equations with applications to grad equations in plasma physics}, Comm. Pure Appl. Math. 76 (2023), 604--615.
	\bibitem{C89} M.G. Crandall, \textit{Semidifferentials, quadratic forms and fully nonlinear elliptic equations of second order}, Ann. Inst. H. Poincar\'e Anal. Non Lin. 6 (1989), 419--435.
	\bibitem{CIL92} M.G. Crandall, H. Ishii and P.L. Lions, \textit{User's guide to viscosity solutions of second order partial differential equations}, Bull. Amer. Math. Soc. 27 (1992), 1--67.	
	\bibitem{DT19} N.M.L. Diehl and R. Teymurazyan, \textit{Reaction-diffusion equations for the infinity Laplacian}, Nonlinear Anal. 199 (2020), 111956, 12 pp.	
	\bibitem{GT01} D. Gilbarg and N.S. Trudinger, \textit{Elliptic partial differential equations of second order}, Classics in Mathematics, Springer, 2001.
	\bibitem{I89} H. Ishii, \textit{On uniqueness and existence of viscosity solutions of fully nonlinear second-order elliptic PDEs}, Comm. Pure Appl. Math. 42 (1989), 15--45.
	\bibitem{I92} H. Ishii, \textit{Perron's method for monotone systems of second-order elliptic partial differential equations}, Differ. Integral Equ. 5 (1992), 1--24.
	\bibitem{IK91} H. Ishii and S. Koike, \textit{Viscosity solutions for monotone systems of second-order elliptic PDEs}, Commun. Part. Diff. Eq. 16 (1991), 1095--1128.
	\bibitem{IK912} H. Ishii and S. Koike, \textit{Viscosity solutions of a system of nonlinear second-order elliptic PDEs arising in switching games}, Funkcial. Ekvac. 34 (1991), 143--155.
	\bibitem{KR97} P. Koskela and S. Rodhe, \textit{Hausdorff dimension and mean porosity}, Math. Ann. 309 (1997), 593--609.
	\bibitem{MNS19} E. Moreira dos Santos, G. Nornberg and N. Soave, \textit{On unique continuation principles for some elliptic systems}, Ann. Inst. H. Poincar\'e Anal. Non Lin\'eaire 38 (2021), 1667--1680.
	\bibitem{MR20} A. Miranda and J.D. Rossi, \textit{A game theoretical approach for a nonlinear system driven by elliptic operators}, SN Partial Differ. Equ. Appl. 1 (2020), Paper No. 14, 41pp.
    \bibitem{P1} D. Phillips, \textit{A minimization problem and the regularity of solutions in the presence of a free boundary}, Indiana Univ. Math. J. 32 (1983), 1--17.
    \bibitem{P2} D. Phillips, \textit{Hausdorff measure estimates of a free boundary for a minimum problem},
    Comm. Partial Differential Equations 8 (1983), 1409--1454.
	\bibitem{S05} L. Silvestre, \textit{The two membranes problem}, Comm. Partial Differential Equations 30 (2005), 245--257.
	\bibitem{SS14} L. Silvestre and B. Sirakov, \textit{Boundary regularity for viscosity solutions of fully nonlinear elliptic equations}, Comm. Partial Differential Equations 39 (2014), 1694--1717.
	\bibitem{ST18} A. Salda\~na and H. Tavares, \textit{Least energy nodal solutions of Hamiltonian elliptic systems with Neumann boundary conditions}, J. Differential Equations 265 (2018), 6127--6165.
	\bibitem{T16} E. Teixeira, \textit{Regularity for the fully nonlinear dead-core problem}, Math. Ann. 364 (2016), 1121--1134.	
	\bibitem{TU14} E. Teixeira and J.M. Urbano, \textit{An intrinsic Liouville theorem for degenerate parabolic equations}, Arch. Math. (Basel) 102 (2014), 483--487.
	\bibitem{WY22} Y. Wu and H. Yu, \textit{On the fully nonlinear Alt-Philips equation}, Int. Math. Res. Not. IMRN 2022, no. 11, 8540--8570.
	\bibitem{Z85} E. Zeidler, \textit{Nonlinear functional analysis and its applications I - Fixed-point theorems}, Springer-Verlag, New York, 1986. xxi+897 pp.
\end{thebibliography}
\end{document}